\documentclass[11pt,a4paper]{amsart}

\usepackage[all]{xy}
\usepackage{amssymb}

 \theoremstyle{plain}

\newtheorem{teo}{Theorem}
\newtheorem{lem}[teo]{Lemma}
\newtheorem{cor}[teo]{Corollary}

\newtheorem*{thmA}{Theorem A}
\newtheorem*{thmB}{Theorem B}
\newtheorem*{thmC}{Theorem C}

\theoremstyle{definition}

\newtheorem{eje}{Example}

\newcommand{\argu}{\hbox to 7truept{\hrulefill}}

\DeclareMathOperator{\Zen}{Z}

\newcommand{\N}{\mathbb{N}}
\newcommand{\Z}{\mathbb{Z}}
\newcommand{\Q}{\mathbb{Q}}

\newcommand{\Span}{\text{Span}}

\keywords{Powerful $p$-groups, uniform pro-$p$ groups}

\subjclass{20}

\begin{document}

\title{A characterization of powerful $p$-groups}

\author{Jon Gonz\'alez-S\'anchez}

\address{
Departamento de Matem\'aticas\\
   Facultad de Ciencia y Tecnolog\'{\i}a\\
   Universidad del Pa\'is Vasco - Euskal Herriko Unibertsitatea UPV/EHU\\
   Spain}
\email{jon.gonzalez@ehu.es}
\author{Amaia Zugadi-Reizabal}

\address{
Departamento de Did\'actica de las Matem\'aticas y de las Ciencias Experimentales\\
   Escuela Universitaria de Magisterio de Bilbao\\
   Universidad del Pa\'is Vasco - Euskal Herriko Unibertsitatea UPV/EHU\\
   Spain}
\email{amaia.zugadi@ehu.es}

\thanks{The authors were supported by IKERTALDE programme of the Basque Government, grant IT753-13.  J. Gonz\'alez-S\'anchez was supported by grant MTM2011-28229-
C02-01; and  A. Zugadi-Reizabal by grant MTM2011-28229-C02-02 from the Spanish Ministry of Economy and Competitivity. J. Gonz\'alez-S\'anchez also acknowledges support through the Ram\'on y Cajal Programme of the Spanish Ministry of Science and Innovation. }

\begin{abstract}
In \cite{KS} Benjamin Klopsch and Ilir Snopce posted the conjecture that for $p\geq 3$ and $G$ a torsion-free pro-$p$ group $d(G)=\dim (G)$ is a sufficient and necessary condition for the pro-$p$ group $G$ to be uniform. They  pointed out that this follows from the more general question of whether for a finite $p$-group $d(G)=\log_p(|\Omega_1(G)|)$ is a sufficient and necessary condition for the group $G$ to be powerful. In this short note we will give a positive answer to this question for $p\geq 5$.
\end{abstract}

\maketitle

\section{Introduction}

We say that a finite $p$-group is powerful if $[G,G]\leq G^p$ if $p$ is odd or $[G,G]\leq G^4$ if $p=2$. Powerful $p$-groups were first introduced by Lubotzky and Mann in \cite{LuMa} and they have played a prominent role in studying finite $p$-groups and compact $p$-adic analytic groups. By a result of Lazard, a compact topological group $G$ has an analytic structure over $\Q_p$ if and only if it has a finite index subgroup $U$ which is finitely  generated, powerful and torsion-free \cite{DDMS}. A finitely generated torsion-free powerful pro-$p$ group is what we call a uniform group. A uniform group has a natural Lie algebra structure which is given by the inversion of the Baker-Campbell-Hausdorff formula and the additive structure of the Lie algebra defines the structure of $\Q_p$-variety.

In \cite{KS} Klopsch and Snopce have conjectured that for an odd prime $p$ a torsion-free pro-$p$ group is uniform if and only if the minimal number of generators and the dimension coincide. They proved this conjecture in the special case when the group is solvable leaving the general case open. They also pointed out that this conjecture follows from the more general question of whether for a finite $p$-group $G$ and $p$ an odd prime $d(G)=\log_p(|\Omega_1(G)|)$ is a sufficient and necessary condition for the group $G$ to be powerful. This question turned out to be false for the prime $p=3$ (see Example \ref{ejemplo}). The aim of this paper is to give a positive answer to this question for $p\geq 5$, and provide a counterexample for the prime $p=3$.

\begin{thmA}
Let $p\geq 5$ and let $G$ be a finite $p$-group . Then the following conditions are equivalent:
\begin{enumerate}
\item $G$ is powerful,
\item $d(G)=\log_p (|\Omega_1(G)|)$.
\end{enumerate}
\end{thmA}

In a finite $p$-group $G$ the minimal number of generators coincides with $\log_p(|G:G^p[G,G]|)$. Therefore one can refrase Theorem A by saying that $|\Omega_1(G)|=|G:G^p[G,G]|$ is a necessary and sufficient condition for a $p$-group $G$ to be powerful. In \cite{HL} Hethelyi an Levai proved that for a powerful $p$-group $G$, one has $|G:G^p|=|\{g\in G\mid g^p=1\}|$, therefore this result can be seen as a converse of this fact. Theorem B proves Theorem A and it also includes a weaker version for the prime $p=3$.
The main ingredient for proving Theorem B are the so-called $\omega$-maximal groups, which were introduced in \cite{GK} by Gonz\'alez-S\'anchez and Klopsch.

\begin{thmB}
Let $p$ be an odd prime, $G$ a finite $p$-group and let $k\leq p-2$ and $i\geq 1$ or $k=p-1$ and $i\geq 2$. Then the following conditions are equivalent:
\begin{enumerate}
\item $\gamma_k(G)\leq G^{p^i}$,
\item $|G:G^{p^i}\gamma_k(G)|=|\Omega_{\{ i\}}(G)|$.
\end{enumerate}
\end{thmB}

Theorem B is false for $k=p-1$ and $i=1$. We will construct an infinite family of counterexamples to this fact.

\begin{thmC}
Given an odd prime $p$ and a positive integer $s\geq p+1$, there exists a finite $p$-group $G$ such that
\begin{enumerate}
\item $|G|=p^s$,
\item $G$ is of maximal class,
\item $|G:G^p\gamma_{p-1}(G)|=|\Omega_1(G)|$,
\item $\gamma_{p-1}(G)\nleq G^p$.
\end{enumerate}
\end{thmC}

The text is organized as follows: In Section 2 we recall the main properties of $\omega$-maximal groups; and in Section 3 we prove Theorem B and construct the family of counterexamples that prove Theorem C.

\textit{Notation}: $\Omega_{\{ i\}}(G)$ denotes the subset of elements of order $p^i$ and $\Omega_i(G)$ the subgroup generated by the elements of order $p^i$. $G^{p^i}$ will denote the subgroup generated by the $p^i$-powers of $G$. The rest of the notation is standard in group theory.

\section{$\omega$-maximal groups and interchangeable words}

Let $G$ be a finite group and $\omega (x_1,\ldots ,x_n)\in F (X)$ an element in the free group over the set $X=\{ x_1,\ldots ,x_n\}$. The verbal subgroup $\omega (G)$ is defined as the subgroup generated 
by the subset of $G$, 
$$\{\omega (g_1,\ldots ,g_n)\mid g_1,\ldots ,g_n\in G\}.$$ 
We say that $G$ is $\omega$-maximal if for all proper subgroups $H$ of $G$, $|H:\omega (H)|< |G:\omega (G)|$. The theory of $\omega$-maximal groups was studied in \cite{GK}, but the concept was already present in \cite{Th69}, \cite{La73} and \cite{Ka91}. An important family of words are the so called interchangeable words. We say that a word $\omega$ is interchangeable in a group $G$ if for every normal subgroup $N$ of $G$, one has that
$$[\omega (N),G]\leq [N,\omega (G)]\cdot [\omega (G),G]^p\cdot [\omega (G),G,G].$$
We will need the following lemma about interchangeable words.

\begin{lem}
  \label{swap}
  Let $p$ be an odd prime and $G$ a finite $p$-group, and let $\omega$ be equal to one of the
  group words
  \begin{enumerate}
  \item[\textup{(i)}] $x^{p^i} [y_1,\ldots ,y_k]$ for some $i, k \in
    \N$ with $k \leq p-1$,
  \item[\textup{(ii)}] $x^{p^i} [y_1,\ldots ,y_{p-1}]^{p^{i-1}}[z_1,\ldots ,z_p]$ for some $i \in \N$
    with $i\geq 2$.
  \end{enumerate}
  Then $\omega$ is interchangeable in $G$.
\end{lem}

\begin{proof}
For (i) see \cite[Lemma 3.1]{GK}. In order to prove (ii), first notice that for 
$$\omega =x^{p^i} [y_1,\ldots ,y_{p-1}]^{p^{i-1}}[z_1,\ldots ,z_p],$$
and for $G$ a finite $p$-group one has that
$$\omega (G)=G^{p^i}\gamma_{p-1}(G)^{p^{i-1}}\gamma_p(G),$$
and for $N$ a normal subgroup of $G$,
$$\omega (N)=N^{p^i}\gamma_{p-1}(N)^{p^{i-1}}\gamma_p(N).$$
We claim that
$$[\omega (G),G]^p[\omega (G),G,G]=[G,G]^{p^{i+1}}[G,G,G]^{p^i}\gamma_{p+1}(G)^p\gamma_{p+2}(G).$$
Indeed, by \cite[Theorem 2.5]{FeGoJa} and since $i\geq 2$,
\begin{eqnarray*}
[G^{p^i},G]& \equiv & [G,G]^{p^i} \pmod {\gamma_{p+1}(G)^p\gamma_{p+2}(G)} \text{ and }\\
{[}\gamma_{p-1}(G)^{p^{i-1}},G] &  \equiv  & \gamma_p(G)^{p^{i-1}} \pmod {\gamma_{p+1}(G)^p\gamma_{p+2}(G)}.
\end{eqnarray*}
Therefore
$$[\omega (G),G]\equiv [G,G]^{p^i}\gamma_p(G)^{p^{i-1}}\gamma_{p+1}(G)\pmod {\gamma_{p+1}(G)^p\gamma_{p+2}(G)},$$
and by P. Hall collection formula, 
$$[\omega (G),G]^p\equiv [G,G]^{p^{i+1}}\gamma_p(G)^{p^i}\pmod {\gamma_{p+1}(G)^p\gamma_{p+2}(G)}.$$
Similarly, by  \cite[Theorem 2.5]{FeGoJa},
\begin{eqnarray*}
[G^{p^i},G,G]& \equiv & [G,G,G]^{p^i} \pmod {\gamma_{p+1}(G)^p\gamma_{p+2}(G)} \text{ and }\\
{[}\gamma_{p-1}(G)^{p^{i-1}},G,G] &  \equiv  & \gamma_{p+1}(G)^{p^{i-1}} \pmod {\gamma_{p+1}(G)^p\gamma_{p+2}(G)}.
\end{eqnarray*}
Therefore, since $\gamma_{p+1}(G)^p\gamma_{p+2}(G)$ is contained in $[\omega (G),G]^p[\omega (G),G,G]$ we have the desired equality
$$[\omega (G),G]^p[\omega (G),G,G]=[G,G]^{p^{i+1}}[G,G,G]^{p^i}\gamma_{p+1}(G)^p\gamma_{p+2}(G).$$
Now, by  \cite[Theorem 2.4]{FeGoJa}, 
\begin{eqnarray*}
[N^{p^i},G] \equiv   [N,G]^{p^i}  \equiv   [N,G^{p^i}] \pmod {\gamma_{p+1}(G)^p\gamma_{p+2}(G)}.
\end{eqnarray*}
By the same theorem and since $[\gamma_p(N),G]\leq [\gamma_p(G),N]$ (see \cite[Lemma 4.9]{Kh}), we deduce
$$[\gamma_{p-1}(N)^{p^{i-1}},G ]\leq    [N,\gamma_{p-1}(G)^{p^{i-1}} ]\gamma_{p+1}(G)^p\gamma_{p+2}(G).$$
Hence one can easily conclude that
$$[\omega (N),G] \leq [N,\omega (G)] \pmod{[\omega (G),G]^p[\omega (G),G,G]},$$
and $\omega$ is interchangeable in $G$.
\end{proof}

The theory of $\omega$-maximal $p$-groups for which $\omega$ is an interchangeable word is very particular and will lead to the proof of the main results of this paper.

\begin{teo}
  \label{thm:thompson}
  Let $\omega$ be a word, and let $G$ be a $\omega$-maximal finite $p$-group
  such that $\omega$ is interchangeable in $G$.  Then one has $\omega (G) \leq
  \Zen (G)$.
\end{teo}

\begin{proof}
See \cite[Theorem 3.3]{GK}.
\end{proof}

\section{The characterization of powerful $p$-groups and uniform pro-$p$ groups}

In this section we give the proof of Theorem B. We start with the case $p\geq 5$.

\begin{teo}
\label{main_p5}
Let $p\geq 5$, $k\leq p-2$ and let $G$ be a finite $p$-group . Then the following conditions are equivalent:
\begin{enumerate}
\item $\gamma_k(G)\leq G^{p^i}$,
\item $|G:G^{p^i}\gamma_k(G)|=|\Omega_{\{i\}}(G)|$.
\end{enumerate}
\end{teo}

\begin{proof}
(1) implies (2) follows from  \cite{Wil} or \cite {GJ}.

For the converse take 
$$\mathcal{C}=\{ H\leq G\mid |H:H^{p^i}\gamma_{k+1}(H)|\geq |G:G^{p^i}\gamma_{k+1}(G)|\}$$ 
and take $K$ a minimal element in $\mathcal{C}$ with respect to inclusion. It follows that for all subgroups $H$ of $K$, $|K:K^{p^i}\gamma_{k+1}(K)|> |H:H^{p^i}\gamma_{k+1}(H)|$. That is, the subgroup $K$ is $\omega$-maximal for the word $\omega =x^{p^i}[y_1,y_2,\ldots ,y_{k+1}]$. By Lemma \ref{swap} we have that the word $\omega$ is interchangeable in $G$. Thus, by Theorem \ref{thm:thompson}, it follows that $K^{p^i}\gamma_{k+1}(K)\subseteq Z(K)$. Therefore  the nilpotency class of the group $K$ is at most $k+1\leq p-1$. In particular the $p$-group $K$ is regular and $|K:K^{p^i}|=|\Omega_{\{i\}}(K)|$. Then we have the following inequalities:

\begin{equation*}
\begin{split}
|G:G^{p^i}\gamma_{k+1}(G)| & \leq  |K:K^{p^i}\gamma_{k+1}(K)|\leq |K:K^{p^i}| \\ 
& =|\Omega_{\{i\}}(K)| \leq  |\Omega_{\{i\}}(G)|= |G:G^{p^i}\gamma_k(G)|\\ 
&\leq |G:G^{p^i}\gamma_{k+1}(G)|.
\end{split}
\end{equation*}
Since the first and the last term are the same all the inequality are equalities. Hence $G^{p^i}\gamma_k(G)=G^{p^i}\gamma_{k+1}(G)$. This implies that $\gamma_k(G)\leq G^{p^i}$.
\end{proof}

As a consequence we can provide a positive answer to the Question 1.9 in \cite{KS} for $p\geq 5$.

\begin{cor}
Let $p\geq 5$ and let $G$ a finite $p$-group. Then the following conditions are equivalent:
\begin{enumerate}
\item $G$ is powerful,
\item $d(G)=\log_p (|\Omega_1(G)|)$.
\end{enumerate}
\end{cor}

\begin{proof}
The proof follows immediately from the case $k=2$ and $i=1$ in the previous theorem.
\end{proof}

And we can solve positively Conjecture 1.1 of \cite{KS} for $p\geq 5$.

\begin{cor}
Let $p\geq 5$ and let $G$ be a torsion-free $p$-adic analytic pro-$p$ group. Then the following conditions are equivalent:
\begin{enumerate}
\item $G$ is uniform,
\item $d(G)=\text{dim} (G)$.
\end{enumerate}
\end{cor}

\begin{proof}
This follows from the previous corollary and Proposition 1.10 of \cite{KS}.
\end{proof}

For proving the case $k=p-1$ we need to use the concept of $i$-regular $p$-groups from the work of Bannuscher (see \cite{Ba81}).  We say that a finite $p$-group $G$ is $i$-regular if for all $x,y\in G$ 
$$x^{p^i}y^{p^i}=(xy)^{p^i}z,$$
where $z\in [\langle x,y\rangle, \langle x,y\rangle]^{p^i}$. We continue with the following lemma.

\begin{lem}
Let $G$ be a finite $p$-group  and  $\omega =x^{p^i} [y_1,\ldots ,y_{p-1}]^{p^{i-1}}[z_1,\ldots ,z_p]$ for some $i \in \N$ with $i\geq 2$. Then for a $\omega$-maximal $p$-group $G$ one has that
$|G:G^{p^i}|=|\Omega_{\{ i\}} (G)|$.
\end{lem}

\begin{proof}
By Lemma \ref{swap}, $\omega$ is interchangeable and by Theorem \ref{thm:thompson} $\omega (G)\leq Z (G)$. That is
$$[G^{p^i},G][\gamma_{p-1}(G)^{p^{i-1}},G]\gamma_{p+1}(G)=1.$$
Notice that by  \cite[Theorem 2.5]{FeGoJa}
\begin{eqnarray*}
[G^{p^i},G] & \equiv & [G,G]^{p^i} \pmod{\gamma_{p+1}(G)} \text{ \ \ and} \\
{[}\gamma_{p-1}(G)^{p^{i-1}},G] & \equiv & \gamma_p(G)^{p^{i-1}}\pmod{\gamma_{p+1}(G)}.
\end{eqnarray*}
In particular we have
$$[G,G]^{p^i}\gamma_{p}(G)^{p^{i-1}}\gamma_{p+1}(G)=1.$$
Therefore by P. Hall collection formula \cite[Chap. III, Theorem 9.4]{Hu},  $G$ is $i$-regular. Hence, the lemma follows from  \cite[Satz 4]{Ba81}.
\end{proof}

\begin{teo}
Let $p$ be an odd prime, $i\geq 2$ and $G$ be a finite $p$-group . Then the following conditions are equivalent:
\begin{enumerate}
\item $\gamma_{p-1}(G)\leq G^{p^i}$,
\item $|G:G^{p^i}\gamma_{p-1}(G)|=|\Omega_{\{i\}}(G)|$.
\end{enumerate}
\end{teo}

\begin{proof}
(1) implies (2) follows from  \cite{HL}, \cite{Wil} or \cite {GJ}.

For the converse put
$$\omega =x^{p^i} [y_1,\ldots ,y_{p-1}]^{p^{i-1}}[z_1,\ldots ,z_p]$$
and take 
$$\mathcal{C}=\{ H\leq G\mid |H:\omega (H)|\geq |G:\omega (G)|\}$$ 
and take $K$ a minimal element in $\mathcal{C}$ with respect to the inclusion. It follows that for all subgroups $H$ of $K$, $|K:\omega (K)|> |H:\omega (H)|$. That is, the subgroup $K$ is $\omega$-maximal. By Lemma \ref{swap} we have that the word $\omega$ is interchangeable in $G$. By the previous lemma one has that $|K:K^{p^i}|=|\Omega_{\{i\}}(K)|$. Then we have the following inequalities:

\begin{equation*}
\begin{split}
|G:\omega (G)| & \leq |K:\omega (K)|\leq |K:K^{p^i}| \\ 
& =|\Omega_{\{i\}}(K)| \leq  |\Omega_{\{i\}}(G)|= |G:G^{p^i}\gamma_{p-1}(G)|\\ 
&\leq |G:\omega (G)|.
\end{split}
\end{equation*}
Since the first and the last term are the same, all inequalities are equalities. Hence $G^{p^i}\gamma_{p-1}(G)=G^{p^i}\gamma_{p-1}(G)^{p^{i-1}}\gamma_p(G)$. This implies that $\gamma_{p-1}(G)\leq G^{p^i}$.
\end{proof}

We finish this section by giving a family of examples showing that Theorem B can not be improved. These examples prove Theorem C.

\begin{eje}
\label{ejemplo}
Take the $\Z_p$-lattice $M=\Span (x_1,\ldots ,x_{p-1})$ of rank $p-1$, and consider the following automorphism $\alpha$ of $M$:
\begin{eqnarray*}
\alpha (x_i) & = & x_{i+1} \ \ \ \text{if $i\leq p-2$} \\
\alpha (x_{p-1}) & = & x_1^{-1}\ldots x_{p-1}^{-1}.
\end{eqnarray*}
$\alpha$ is an automorphism of order $p$ that acts uniserially on $M$. Put $M_1=M$, and $M_r=[M_{r-1},\alpha]$. The semidirect product of $M$ and $\langle \alpha\rangle$ is the unique pro-$p$ group $H$ of maximal class, furthermore, $M$ is a maximal subgroup of $H$ and all elements of $H\setminus M$ are of order $p$ (see \cite[Section 7.4]{LeMc}). Therefore by \cite[Hilfssatz 10.9]{Hu} for $j=1,\ldots ,p-1$, one has that
$$1=(\alpha^j x)^p=x^p\sum_{i=1}^p[x,_i\alpha^j]^{ {p \choose i}}.$$
In particular 
$$x^p\sum_{i=1}^p[x,_i\alpha^j]^{ {p \choose i}}=1.$$
Put $N_r=M/M_r$ and consider $z$ a generator of $M_{r-1}/M_r$. Now, consider $G_r$ the non split extension
$$1\longrightarrow N_r\longrightarrow G_r\longrightarrow C_p\longrightarrow 1,$$
where $C_p=\langle y \rangle$ and the extension is defined by the identity $y^p=z$ and the action of $y$ in $N_r$ is given by $\alpha$. For any element in $x\in N_r$ and for $j=1,\ldots ,p-1$, by the previous identity and \cite[Hilfssatz 10.9]{Hu} we have 
$$(y^jx)^p=y^{jp}x^p\sum_{i=1}^p[x,_iy^j]^{ {p \choose i}}=z^j\neq 1.$$
Therefore $\Omega_1(G_r)=\Omega_1(M/M_r)$, in particular if $r\geq p$, it follows that $|\Omega_1(G_r)|=p^{p-1}$. On the other hand, since $\alpha$ acts uniserially on $M$, we have $|G_r:G_r^p\gamma_{p-1}(G_r)|=p^{p-1}$. But clearly the group $G_r$ does not satisfy the inclusion $\gamma_{p-1}(G_r)\leq G_r^p$.
\end{eje}


\begin{thebibliography}{10}



 \bibitem{Ba81} W. Bannuscher, Eine Verallgemeinerung des RegularitŠtsbegriffes bei p-Gruppen. I. (German) BeitrŠge Algebra Geom. No. 11 (1981), 51Ð63. 
 
 \bibitem{DDMS}{J.~Dixon, M.~du~Sautoy, A.~Mann,  D.~Segal}.
\textit{Analytic pro-$p$ groups}, 2nd ed.. (Cambridge University Press, 1999).
 
 \bibitem{FeGoJa} G.~Fern\'andez-Alcober, J.~Gonz\'alez-S\'anchez,
  A.~Jaikin-Zapirain, Omega subgroups of pro-$p$ groups,
  Israel J.\ Math.\ \textbf{166} (2008), 393--410.
 

\bibitem{GK} J. Gonz\'alez-S\'anchez, B. Klopsch, On $\omega$-maximal groups,
{\it J. Algebra\/} {\bf 328} (2011), 155--166.

\bibitem{GJ}
J. Gonz\'alez-S\'anchez, A. Jaikin-Zapirain,
On the structure of normal subgroups of potent $p$-groups,
{\it J. Algebra\/} {\bf 276} (2004), 193--209.

\bibitem{HL} L. Hethelyi, L. Levai, On elements of order $p$ in powerful $p$-groups, {\it J. Algebra\/} {\bf 270} (2003), 1--6.


\bibitem{Hu}
B. Huppert,
\textit{Endliche Gruppen I\/},
Springer, 1967.


\bibitem{Ka91} B.\ Kahn, The total Stiefel-Whitney class of a
    regular representation, J.\ Algebra \textbf{144} (1991),
  214--247.
  
  \bibitem{Kh}
E.I. Khukhro,
\textit{$p$-Automorphisms of finite $p$-groups\/},
Cambridge University Press, 1998.

  
  \bibitem{KS} B. Klopsch, I. Snopce,  A characterization of uniform pro-$p$ groups (2012).  
arXiv:1210.4965.

  
  \bibitem{La73} T.J.\ Laffey, The minimum number of generators
    of a finite $p$-group, Bull.\ London Math.\ Soc.\ \textbf{5}
  (1973), 288--290.
  
  \bibitem{LeMc} C.~R.~Leedham-Green and S.~McKay.  \textit{The
    structure of groups of prime power order} (London Mathematical
  Society Monographs, New Series \textbf{27}, Oxford University Press,
  Oxford, 2002).

  
  \bibitem{LuMa}  A. Lubotzky, A. Mann. Powerful p-groups. I. Finite Groups. \textit{J. Algebra} \textbf{105} (1987), 484--505.
  
  
  \bibitem{Th69} J.G.\ Thompson, A replacement theorem for
    $p$-groups and a conjecture, J.\ Algebra \textbf{13} (1969),
  149--151.

\bibitem{Wil} L. Wilson, On the power structure of powerful $p$-groups,
\textit{J. Group Theory}, {\bf 5} (2002), 129--144.



\end{thebibliography}
\end{document}